\documentclass{amsart}
\usepackage{amssymb}
\usepackage{verbatim}
\newtheorem{theorem}{Theorem}[section]
\newtheorem{lemma}[theorem]{Lemma}
\newtheorem{proposition}[theorem]{Proposition}

\newtheorem{fact}[theorem]{Fact}
\usepackage[english]{babel} 
\theoremstyle{definition}

\theoremstyle{remark}
\newtheorem{remark}[theorem]{Remark}

\numberwithin{equation}{section}

\renewcommand{\span}{\mathrm{span}}
 
\newcommand{\dist}{\mathrm{dist}}
\newcommand{\conv}{\mathrm{conv}}

\newcommand{\R}{\mathbb{R}}
\newcommand{\E}{\mathbb{E}}
\newcommand{\N}{\mathbb{N}}

\newcommand{\X}{\mathrm{X}}
\renewcommand{\H}{\mathrm{H}}
\newcommand{\Y}{\mathrm{Y}}
\newcommand{\Z}{\mathrm{Z}}
\newcommand{\B}{\mathbf{B}}
\newcommand{\I}{\mathrm{I}}

\renewcommand{\S}{\mathbf{S}}

\renewcommand{\ker}{\mathrm{Ker}}
\renewcommand{\mod}{/}
\newcommand{\1}{\boldsymbol{1}}

\renewcommand{\to}{\rightarrow}
\begin{document}

\title{Convex-transitivity of Banach algebras via ideals}

\begin{abstract}
We investigate a method for producing concrete convex-transitive Banach spaces. The gist of the method is 
in getting rid of dissymmetries of a given space by taking a carefully chosen quotient. The spaces of interest here 
are typically Banach algebras and their ideals. We also investigate the convex-transitivity of ultraproducts and tensor products of Banach spaces. 
\end{abstract}

\author{Jarno Talponen}
\address{Aalto University, Institute of Mathematics, P.O. Box 11100, FI-00076 Aalto, Finland} 
\email{talponen@cc.hut.fi}
\date{\today}
\maketitle

\section{Introduction}

In this paper we study and construct Banach algebras with a rich isometry group, which, in a sense, comes close to  
acting transitively on the unit sphere. To facilitate the discussion, let us recall some basic notions.
We denote the closed unit ball of a Banach space $\X$ by $\B_{\X}$ and the unit sphere of $\X$ by $\S_{\X}$.
A Banach space $\X$ is called \emph{transitive} if for each $x\in \S_{\X}$ the orbit
$\mathcal{G}_{\X}(x)\stackrel{\cdot}{=}\{T(x)|\ T\colon \X\rightarrow \X\ \mathrm{is\ an\ isometric\ automorphism}\}$ is $\S_{\X}$. In other words, the isometry group acts transitively on the unit sphere. If $\overline{\mathcal{G}_{\X}(x)}=\S_{\X}$ (resp. $\overline{\conv}(\mathcal{G}_{\X}(x))=\B_{\X}$) 
for all $x\in\S_{\X}$, then $\X$ is called \emph{almost transitive} (resp. \emph{convex-transitive}). 
It was first reported by Pelczynski and Rolewicz in 1962 \cite{PR} that the space $L^{p}$ is almost transitive for 
$p\in [1,\infty)$ and convex-transitive for $p=\infty$ (see also \cite{Rol}). These concepts are motivated by 
\emph{Mazur's rotation problem} appearing in \cite[p.242]{Ba}, which remains open.
We refer to \cite{BR2} for a survey and discussion on the matter.

By applying categorical methods one can verify the \emph{existence} of a rich class of almost transitive Banach spaces (see e.g. the above survey). However, not so many \emph{concrete} almost transitive or even convex-transitive spaces are known. Specimens of such spaces can be found for example in \cite{Ca5}, \cite{Ca6}, \cite{GJK}, \cite{Rambla} and \cite{RT}. The purpose of this paper is to investigate a method for producing concrete convex-transitive Banach algebras.

In many apparently 'fairly homogenous' Banach algebras there exists an obvious obstruction for the convex-transitivity of the space. For example, in $\ell^{\infty}$, whose isometry group has the anticipated form (\cite[2.f.14]{LTI}), there exist both finitely and infinitely supported vectors. The cardinality of the supports of the vectors are preserved under the isometries, thus $y \notin \overline{\conv}(\mathcal{G}_{\ell^{\infty}}(x))$ for $y\in \ell^{\infty}\setminus c_{0},\ x\in c_{0}$,
and therefore $\ell^{\infty}$ fails to be convex-transitive. Here we consider a natural remedy to the non-convex-transitivity by removing the obstruction caused be the finitely supported vectors, while roughly retaining the structure of the space. Namely, $\ell^{\infty}\mod c_{0}=C(\beta\omega \setminus \omega)$ is a convex-transitive space, as was observed by Cabello (\cite{Ca5}). In fact, this was established mainly from topological considerations, but here we obtain an alternative route to this result, without any reference to $\beta\omega$, by observing that $c_{0}$ is a suitable ideal of $\ell^{\infty}$ (see Section \ref{sect: c_0ideal}).

For another type of example, recall the standard Calkin algebra 
\[\mathcal{C}(\H)=\mathcal{B}(\H)\mod \mathcal{K}(\H),\] 
$\H$ a separable complex Hilbert space, is a quotient of bounded linear operators on $\H$ (i.e. $\mathcal{B}(\H)$)
by the ideal of compact operators, $\mathcal{K}(\H)$. The Calkin algebra is convex-transitive, as was observed by Becerra-Guerrero and Rodriguez-Palacios in \cite{BR_manuscripta}. In $\mathcal{B}(\H)$ there appear intuitively two 'levels' of operators: the operator norm closure of finite-rank operators, i.e. the ideal of compact operators $\mathcal{K}(\H)$, and the $\mathrm{SOT}$-closure of these operators, i.e. the ideal of all operators $\mathcal{B}(\H)$. Clearly, in $\mathcal{B}(\H)$ the compact operators are not suitable for approximating general operators in the operator norm but in passing to the Calkin algebra this particular problem ceases to exist. An optimistic philosophy behind this example is that dividing out the compact operators homogenizes the space $\mathcal{B}(\H)$, as there will be in a sense only a single 'level' left, and the high degree of symmetry of the Hilbert space takes care of the rest in achieving  convex-transitivity.

For general Banach algebras the situation is different, as there can exist more natural ideals and less 
symmetry. Still, the convex-transitive examples $\mathcal{C}(\H)$ and $\ell^{\infty}\mod c_{0}$ are fairly different. This suggests that the described rationale for the existence of convex-transitive spaces could be applicable in various settings.
 
In fact, it turns out that this kind of approach is fruitful in connection with surprisingly many kinds of unital Banach algebras. The main results here are about constructing convex-transitive Banach spaces by taking a quotient with respect to a carefully chosen subspace, which is typically an ideal. We will study symmetries of function spaces on some known topological spaces and on infinite trees. We investigate the symmetries of Banach algebras with ideals defined in terms of ultrafilters and ideals on sets. We will give examples of convex-transitive Corona type algebras related to 
the previous Calkin algebra example. At the end of the paper we comment on the symmetries of projective tensor products. To conclude, by using our method, we obtain various natural examples of convex-transitive spaces, and in the constructions we also observe some connections to other branches of mathematics, such as Ergodic theory, or ideals on regular cardinals. 

\subsection{Preliminaries}
We denote by $\X$, $\Y$ and $\Z$ real Banach spaces. The closed unit ball and the unit sphere are denoted by 
$\B_{\X}$ and $\S_{\X}$, respectively. We denote by 
\[\mathcal{G}_{\X}\stackrel{\cdot}{=}\{T|\ T\colon \X\rightarrow \X\ \mathrm{is\ an\ isometric\ automorphism}\}\] 
the isometry group of $\X$. We refer to \cite{fhhmz}, \cite{DU}, \cite{En}, \cite{FJ1}, \cite{FJ2}, \cite{Heinrich}, 
\cite{Kechris} and \cite{Rol} for suitable background information. 

Given a set $\Gamma$ we denote by $\mathcal{P}(\Gamma)=\{A\subset \Gamma\}$ the power set of $\Gamma$.
Given a compact space $K$ we will denote by $\mathrm{M}(K)$ the space of regular Borel measures with the total variation norm. If $f\in L^{p}=L^{p}([0,1],m)$, where $m$ is the Lebesgue measure and $A\subset [0,1]$ is measurable set with 
$m(A)>0$, we write $\E(f|A)=\frac{1}{m(A)}\int_{A}f\ dm$.

We denote by $\mathcal{B}(\X,\Y)$ the space of bounded linear operators $\X\to \Y$. We write 
$\mathcal{B}(\X)=\mathcal{B}(\X,\X)$ and $\I$ is the identity operator. 
If $\X$ is a Banach algebra, we sometimes consider $\mathcal{B}(\X)$ as a left $\X$-Banach module, where 
the operation $x\cdot T$ is given by $(x\cdot T)[y]=xT[y]$ for $T\in \mathcal{B}(\X)$ and $x,y\in \X$.
In the case where $\X$ is a function space, say, on $\Omega$, we will adopt the notation $1_{A}\cdot \I$, 
where $\I$ is the identity mapping and $1_{A}$ is the characteristic function (or indicator function) supported on 
$A\subset \Omega$ in a sensible way depending on the setting. 

As customary, we will denote by $\mathcal{C}(\X)$ the Calkin type of Banach algebra 
$\mathcal{B}(\X)\mod \mathcal{K}(\X)$. We denote by $\mathcal{S}(\X)\subset \mathcal{B}(\X)$ the ($2$-sided) ideal of operators with a separable range.

Given a Banach space $\X$ with a locally convex topology $\tau$ and a subspace $\Y\subset\X$, we sometimes consider 
$\X\mod \Y$ in topology $\tau$. This is a slight abuse of notation and is to be understood in the sense that $\X\mod \Y$ has a locally convex topology consisting of sets of the form $U+\Y,\ U\in \tau$. Of course, the subspace $\Y$ 
need not be $\tau$-closed, a priori, and in this case the singletons in $\X\mod \Y$ are not $\tau$-closed.

An example of a locally convex topology on $\mathcal{B}(\X)$ appearing here is the \emph{topology of uniform convergence on separable subspaces}. This topology $\tau$ is generated by a basis of sets of the following form:
\[\{T\in \mathcal{B}(\X):\
\|T_{1}|_{\Y}-T|_{\Y}\|_{\mathcal{B}(\Y,\X)},\|T_{2}|_{\Y}-T|_{\Y}\|_{\mathcal{B}(\Y,\X)},\ldots,\|T_{n}|_{\Y}-T|_{\Y}\|_{\mathcal{B}(\Y,\X)}<\epsilon\},\]
where $T_{1},T_{2},\ldots,T_{n}\in \mathcal{B}(\X)$, $\Y\subset\X$ is a separable subspace and $\epsilon>0$.
For example, in $\mathcal{B}(\ell^{p}(\omega_{1})),\ 1\leq p<\infty,$ the operators $x\mapsto 1_{[0,\alpha)}\cdot x$ 
converge to $\I$ as $\alpha\to \omega_{1}$ in the above topology $\tau$, but not in the operator norm topology. If $\X$ is separable itself, then $\tau$ coincides with the operator norm topology. The \emph{strong operator topology} ($\mathrm{SOT}$) on $\mathcal{B}(\X,\Y)$ is the topology inherited from $\Y^{\X}$ with the product topology.

Suppose that $\Gamma$ is a set, $\mathcal{I}$ is an ideal on $\Gamma$, 
$\{x_{\alpha}: \alpha \in \Gamma\}$ is a subset of a Hausdorff topological space $X$ and 
$x\in X$. If for each open neighbourhood $U$ of $x$ it holds that 
$\{\alpha\in \Gamma:\ x_{\alpha}\notin U\}\in \mathcal{I}$, then we say that $x_{\alpha}$ converges to $x$ with respect 
to the ideal $\mathcal{I}$, or $\lim_{\alpha,\mathcal{I}}x_{\alpha}=x$ in short. 
In this case, given $\mathcal{F}=\{\Gamma \setminus I:\ I\in \mathcal{I}\}$, the dual filter of $\mathcal{I}$, we 
alternatively denote $\lim_{\alpha,\mathcal{F}}x_{\alpha}=x$.

If $T$ is a topological space we denote by $CB(T)$ the space of continuous bounded functions on $T$ with the sup norm.
Given an ideal $\mathcal{I}$ on $T$, we denote by $CB_{0}(T,\mathcal{I})\subset CB(T)$ the closed subspace of functions $f$ 
such that $\lim_{t,\mathcal{I}}|f(t)|=0$. This is truly a closed subspace, indeed, fix
$u\in \overline{CB_{0}(T,\mathcal{I})}$ and $\epsilon>0$. Let $v\in CB_{0}(T,\mathcal{I})$ such that 
$\|u-v\|<\epsilon/2$ and $I\in \mathcal{I}$ such that $|v(t)|< \epsilon/2$ for $t\in T\setminus I$.
By the triangle inequality $|u(t)|< \epsilon$ for $t\in T\setminus I$ and thus $\lim_{t,\mathcal{I}}u(t)=0$.

An element $x\in \S_{\X}$ is called a \emph{big point} if $\overline{\conv}(\mathcal{G}(x))=\B_{\X}$. Given a locally convex topology $\tau$ on $\X$, the space $\X$ is said to be $\tau$-convex-transitive if $\overline{\conv}^{\tau}(\mathcal{G}(x))=\B_{\X}$ for each $x\in \S_{\X}$.

In order to check the convex-transitivity of a vector-valued function space one often wishes to run 
approximations by convex combinations simultaneously in infinitely many coordinate spaces. In such a
case a uniformity for the convergence of the approximations is required. For this purpose the 
uniformly convex-transitive Banach spaces were introduced in \cite{RT}, and they are defined as follows.
Provided that the space $\X$ under discussion is understood, we denote
\[C_{n}(x)=\left\{\sum_{i=1}^{n}a_{i}T_{i}(x)|\ T_{1},\ldots,T_{n}\in \mathcal{G}_{\X},\ a_{1},\ldots,a_{n}\in [0,1],\ \sum_{i=1}^{n}a_{i}=1\right\}\]for $n\in \N$ and $x\in \S_{\X}$.
We call a Banach space $\X$ \emph{uniformly convex-transitive} if for each $\varepsilon>0$ there
exists $n\in \N$ satisfying the following condition: 
\begin{equation}\label{eq: xydist}
\dist(y,C_{n}(x))\leq \varepsilon\quad \mathrm{for\ all}\ x,y\in \S_{\X},
\end{equation}
that is,
\[\lim_{n\rightarrow \infty}\sup_{x,y \in \S_{\X}}\dist(y,C_{n}(x))=0.\]
For each $\varepsilon>0$ we denote by $K(\varepsilon)$ the least integer $n$, which satisfies \eqref{eq: xydist} 
and such $K(\varepsilon)$ is called \emph{the constant of uniform convex transitivity of $\X$.}

The following elementary fact is applied here frequently.
\begin{fact}\label{Fact1}
For each $x,y,z\in \S_{\X}$ we have $\dist(x,C_{n+m}(z))\leq \dist(x,C_{n}(y))+\dist(y,C_{m}(z))$.
\end{fact}

Recall that a contractive linear projection $P\colon \X\to \Y$ is called an \emph{isometric reflection projection}
if $\I-2P\in \mathcal{G}_{\X}$, or equivalently, if $\|\I-2P\|=1$. If $\Z\subset \X$ is a subspace such that 
there is an isometric reflection projection $P\colon \X^{**}\to \Z^{\bot\bot}$, then $\Z$ is a \emph{$u$-ideal} of $\X$.
For discussion on $u$-ideals we refer to \cite{GKS}, where they were first introduced. 

\section{Some concrete convex-transitive Banach algebras}

Recall that the Sorgenfrey line $S$ is the set $\R$ with the topology generated by half open intervals $[a,b)\subset \R$.
Observe that these intervals are clopen and thus $S$ is not locally compact.
\begin{theorem}\label{Sorgenfrey}
Let $S$ be the Sorgenfrey line, $n\in\N$, and let $\mathcal{I}$ be the ideal of sets $I\subset S^{n}$ such that $I\subset [-k,k]^{n}$ for some $k\in \N$. Then $CB_{0}(S^{n},\mathcal{I})$ is uniformly convex-transitive. 
\end{theorem}

\begin{lemma}\label{lm: clopen}
Let $T$ be a Hausdorff, $0$-dimensional topological space. Let $f\colon T\to \R$ be a continuous function and $a,b\in \R,\ a<b$. Then there exists a clopen set $C\subset T$ such that $f^{-1}((a,\infty))\supset C\supset f^{-1}([b,\infty))$.
\end{lemma}
\begin{proof}
Let $\beta T$ be the Czech-Stone compactification of $T$. Since $T$ is completely regular, we have that $T$ is embedded densely in $\beta T$. Denote by $\tilde{f}\colon \beta T\to \R$ the unique continuous extension of 
$f$. Let $K$ be the closure of $ f^{-1}([b,\infty))$ in $\beta T$. By using the fact that $T$ is $0$-dimensional fix a set $\mathcal{F}$ of clopen subsets of $T$ such that 
$f^{-1}([\frac{a+b}{2},\infty))\subset \bigcup \mathcal{F}\subset f^{-1}((a,\infty))$. By studying $\tilde{f}$ we observe that $\overline{T\setminus \bigcup\mathcal{F}}\cap K=\emptyset$. 
Thus $K\subset \overline{\bigcup\mathcal{F}}$. 

Given $A\in \mathcal{F}$, the characteristic function $1_{A}\colon T\to \{0,1\}$ is continuous and admits a continous extension $\beta T\to \R$. In fact the image of this extension 
is still $\{0,1\}$, since $T$ is dense in $\beta T$. Thus $\overline{A}$ is clopen in $\beta T$.  

Consequently, the set $\{\overline{A}\subset \beta T:\ A\in \mathcal{F}\}$ is an open cover of $K$. Thus the compactness $K$ yields that there exists a finite set $\mathcal{F}_{0}\subset\mathcal{F}$
such that $\{\overline{A}\subset \beta T:\ A\in \mathcal{F}_{0}\}$ is an open cover of $K$. Since the sets $A\in \mathcal{F}_{0}$ are already closed in $T$, we have that 
$f^{-1}([b,\infty))\subset \bigcup \mathcal{F}_{0}$. Now $\bigcup\mathcal{F}_{0}$ is clopen in $T$ as a finite union of clopen sets. 
Note that $\bigcup\mathcal{F}_{0}\subset f^{-1}((a,\infty))$ by the construction of $\mathcal{F}$.
\end{proof}

\begin{proof}[Proof of Theorem \ref{Sorgenfrey}]
Fix $f_{0},g_{0}\in \S_{CB_{0}(S^{n},\mathcal{I})}$ and $\epsilon>0$. We may assume without loss of generality, possibly by multiplying $f_{0}$ with $-1$, that $\sup_{t\in S^{n}}f_{0}(t)=1$. Then one can find intervals $[a_{i},b_{i}), [c_{i},d_{d})\subset S,\ 1\leq i\leq n,$ such that $|g_{0}(t)|<\epsilon$ for $t\in S^{n}\setminus \prod_{i}[a_{i},b_{i})$
and $f_{0}(s)>1-\epsilon$ for $s\in \prod_{i}[c_{i},d_{i})$. 

Note that $\prod_{i}[a_{i},b_{i})\subset S^{n}$ is a clopen subset and thus $\1_{\prod_{i}[a_{i},b_{i})}\in  \S_{CB_{0}(S^{n},\mathcal{I})}$. Let $h_{i}\colon S\rightarrow S$ be homeomorphisms such that 
$h_{i}([a_{i},b_{i}))=[c_{i},d_{i})$ for $1\leq i\leq n$. Putting $h=\prod_{i=1}^{n}h_{i}$ defines a homeomorphism $S^{n}\to S^{n}$. Since $\prod_{i}[a_{i},b_{i})$ is a clopen set, we may define 
rotations $T_{1},T_{2}\in \mathcal{G}_{CB_{0}(S^{n},\mathcal{I})}$ by defining $T_{1}(f)=f\circ h$ and $T_{2}(f)(t)=(1_{\prod_{i}[a_{i},b_{i})}(t)-1_{S^{n}\setminus \prod_{i}[a_{i},b_{i})}(t))(f\circ h)(t)$ for 
$f\in CB_{0}(S^{n},\mathcal{I})$ and $t\in S^{n}$. Now $\|1_{\prod_{i}[a_{i},b_{i})}-\frac{T_{1}(f_{0})+T_{2}(f_{0})}{2}\|<\epsilon$. Since $\epsilon$ was arbitrary, we conclude that 
$1_{\prod_{i}[a_{i},b_{i})}\in \overline{\conv}\mathcal{G}_{CB_{0}(S^{n},\mathcal{I})}(f)$.

Finally, we will apply Lemma \ref{lm: clopen} as follows. Note that if $C\subset S^n$ is any clopen set, then 
the mapping $S^n \to \{-1,1\},\ t\mapsto 1_{S^n \setminus C}-1_{C}$ is continuous. 

By combining the proofs of Lemma \ref{lm: clopen} and \cite[Lemma 2.3]{RT} recursively
one can find for each $k\in \N$ a partition of $S^n$ by disjoint clopen sets $\{C_{i}:\ 0\leq i\leq 2k-1\}$ 
such that $g_{0}(C_{i})\subset [-1+(2i-1)/2k,-1+(2i+3)/2k]$ for $0\leq i\leq 2k-1$. Similarly as in 
the proof of \cite[Thm. 2.4]{RT} we put $B_{i}=\bigcup_{j\geq i}C_{j}$ and
define $R_{i}\in \mathcal{G}_{CB_{0}(S^n , \mathcal{I})}$ by $R_{i}(f)(t)=(1_{B_{i}}(t)-1_{B_{i}^{c}}(t))f(t)$ for $0\leq i\leq 2k-1,\ 
t\in S^n$. Then 
\[\left\|g_{0}-\frac{1}{2k}\sum_{i=0}^{2k-1}R_{i} 1_{\prod_{i}[a_{i},b_{i})}\right\|\leq \frac{1}{2k}.\]
Thus we have the claim according to Fact \ref{Fact1}.
\end{proof}

\begin{remark}
Similar arguments as employed in the proof of Theorem \ref{Sorgenfrey} yield that $C(K)$ is convex-transitive for Alexandroff's double-arrow space $K$. 
\end{remark}

Recall that the double-arrow space $K=[0,1]\times \{0,1\}$ is a closed subspace of $[0,1]^{2}$ in the order topology,
where the order is the lexicographic order. Then $K$ is non-metrizable, separable, compact and a typical counterexample space in topology. This example could be of interest because the corresponding $C(K)$ space has some peculiar properties as a Banach space. Namely, it has been previously studied extensively in connection with some branches of functional analysis other than the isometric theory, for example, in the study of non-separable Banach spaces and approximation theory, see e.g. \cite{Zizler}, \cite{Vlasov}. The fact that $C(K)$ is convex-transitive can also be recovered from \cite{Ca5} and \cite{RT}.

\subsection{Functions on infinite trees}\label{sect: tree}

Here we consider infinite trees $T=\kappa^{<\lambda}\cup \kappa^\lambda$ where $\kappa$ is a cardinal and $\lambda$ is 
a limit ordinal. This has a natural partial order given by $x\leq y$ if $x$ is an initial segment of $y$. 
We endow $T$ with the order topology. Note that $T$ may be non-compact.
We will study spaces $CB(T)$ and $CB_{0}(T,\mathcal{I})$ where $\mathcal{I}$ is an ideal on $T$ given by $\mathcal{I}=\bigcup_{\alpha<\lambda}\mathcal{P}(\kappa^{<\alpha})$.

For example, in the case where $\kappa=2$ and $\lambda=\omega$, $\kappa^\lambda =2^{\omega}$ is homeomorphic to the Cantor set $\Delta$. 
The relevance to this discussion is that $C(\Delta)=C(T)\mod CB_{0}(T,\mathcal{I})$ isometrically,
and that $C(\Delta)$ is a classical example of a convex-transitive space (\cite{Rol}).

We will obtain a method for producing other concrete examples of convex-transitive spaces, namely, 
$C(k^{\omega})$ is convex-transitive for $k<\omega$. This is obtained by modifying the proof of Theorem \ref{thm: trees}.
The similar statement holds for continuous bounded functions on the Baire space $\omega^{\omega}$, as the following result shows.
\begin{theorem}\label{thm: trees}
Let $\kappa$ be an infinite cardinal and $\lambda$ be a limit ordinal such that $|\kappa^{<\lambda}|=|\kappa|$.
Suppose that $T=\kappa^{<\lambda}\cup \kappa^{\lambda}$ is topologized similarly as above and let $\mathcal{I}=\bigcup_{\alpha<\lambda}\mathcal{P}(\kappa^{\alpha})$.
Then the space
\[CB(\kappa^{\lambda})=CB(T)\mod CB_{0}(T,\mathcal{I})\]
is uniformly convex-transitive.
\end{theorem}

\begin{proof}[Proof of Theorem \ref{thm: trees}]
Let $\tilde{f},\tilde{g}\in CB(T)\mod CB_{0}(T,\mathcal{I}),\ \|g\|\leq \|f\|=1$. 
Fix representatives $f,g\in CB(T)$, respectively. Without loss of generality we may assume that 
$\sup_{t\in \kappa^{\lambda}} f(t)=1$. Given $\epsilon>0$ there exists according to the definition of the order topology of $T$ an ordinal $\alpha<\lambda$ and an element $s\in \kappa^{\alpha}$ such that $f(t)>1-\epsilon$ for each extension $t\geq s$.

Let $\phi_{1}\colon T\setminus \{1\} \to T\setminus \kappa^{<\alpha+1}$ be an order preserving 
homeomorphism such that $\phi_{1}(\{(1,k):\ k\in \kappa\})=\{t\in \kappa^{\alpha+1}:\ s\not\leq t\}$ and 
$\phi_{1}(\kappa\setminus \{1\})=\{t\in \kappa^{\alpha+1}:\ s\leq t\}$. Indeed, here we apply the assumption that 
$|\kappa^{<\lambda}|=|\kappa|$. Similarly, we define homeomorphisms
$\phi_{i}\colon T\setminus \{i\} \to T\setminus \kappa^{<\alpha+1}$ for $i<\omega$ 
such that $\phi_{i}(\{(i,k):\ k\in \kappa\})=\{t\in \kappa^{\alpha+1}:\ s\not\leq t\}$ and 
$\phi_{i}(\kappa\setminus \{i\})=\{t\in \kappa^{\alpha+1}:\ s\leq t\}$. 

It is easy to verify that $x+ CB_{0}(T,\mathcal{I}) \mapsto x\circ \phi_{i} + CB_{0}(T,\mathcal{I}),\ x\in CB(T),$ 
defines an isometric automorphism on $CB(T)\mod CB_{0}(T,\mathcal{I})$ for each $i$. Note that
\begin{equation}\label{eq: phiit} 
\frac{1}{m} \sum_{i=1}^{m}f\circ \phi_{i}(r)\geq 1-\epsilon - 2/m\quad \mathrm{for}\ r\in T\setminus \kappa,\ m<\omega.
\end{equation}
Since the selection of $\epsilon$ was arbitrary, we conclude that 
\[1_{T\setminus \kappa}+CB_{0}(T,\mathcal{I})\in \overline{\conv}(\mathcal{G}_{CB(T)\mod CB_{0}(T,\mathcal{I})}(\tilde{f})).\]
Finally, recalling Fact \ref{Fact1}, it suffices to check that 
\begin{equation}\label{eq: g1CB}
g\in \overline{\conv}(\{T(1_{CB(T)}):\ T\in \mathcal{G}_{CB(T)},\ T(CB_{0}(T,\mathcal{I}))=CB_{0}(T,\mathcal{I})\}).
\end{equation}
By studying the behaviour of $g$ at the limit nodes we observe that for each $\epsilon>0$ the tree $T$ can be decomposed into a family of compact clopen subsets $\{C_{\gamma}\}_{\gamma}$ such that 
$\max_{t\in C_{\gamma}}g(t)-\min_{t\in C_{\gamma}}g(t)<\epsilon$ for each $\gamma$. The claim \eqref{eq: g1CB} is accomplished by applying an average of continuous changes of signs similarly as in Theorem \ref{Sorgenfrey}.
\end{proof}

\section{Large banach algebras, ideals and convex-transitivity} 

Perhaps 'Sweeping problems under the rug' would have served as a more intuitive title for this section, as we wish to
cover some undesirable parts of the space in an orderly fashion. Namely, our aim is to fade dissymmetries of the space by taking a quotient with respect to a suitable ideal. The ideal must be chosen carefully in such a way that, on one hand, it covers the sources of dissymetry that one wishes to get rid of, and, on the other hand, dividing out the ideal does not spoil the existing symmetries.

\subsection{Banach ideals from ideals on sets.}\label{sect: c_0ideal}

\begin{theorem}\label{thm: convex_ideal}
Let $\X$ be a uniformly convex-transitive Banach space and let $\kappa$ be an infinite regular cardinal. 
Let $\mathcal{I}$ be an ideal on $\kappa$ such that $\mathcal{I}\subsetneq \mathcal{P}(\kappa)$ and each subset $A\subset \kappa$ with $|A|<\kappa$ is included in $\mathcal{I}$. Let us assume that $\mathcal{I}$ has the following 
homogeneity properties: For each $A,B\subset \kappa,\ A,B\notin \mathcal{I},$ and $\epsilon>0$ there are
$\phi_{i}\colon \kappa\to \kappa,\ 1\leq i\leq n,$ such that 
\begin{enumerate}
\item[(i)]{$\phi_{i}^{-1}(I)\in \mathcal{I}$ if and only if $I\in \mathcal{I}$ for $1\leq i\leq n$.} 
\item[(ii)]{$\frac{1}{n}\sum_{i=1}^{n}1_{\phi_{i}(B)}(x)>1-\epsilon\ \mathrm{for}\ x\in A$.}
\end{enumerate}
Then the Banach space
\[\ell^{\infty}(\kappa,\X)\ \mod\ \{(x_{\alpha})_{\alpha<\kappa}\in \ell^{\infty}(\kappa,\X):\ \lim_{\alpha,\mathcal{I}}\|x_{\alpha}\|= 0\}\] 
is uniformly convex-transitive.
\end{theorem}

Let us make a few comments before the proof. Recall that $\ell^{\infty}$ itself is far from being convex-transitive in view of the characterization of its onto isometries (see \cite[2.f.14]{LTI}) as we discussed in the introduction. As regards the title of this paper, we note that if $\X$ should be additionally a Banach algebra, then $\ell^{\infty}(\kappa,\X)$ becomes, with the pointwise multiplication, a Banach algebra and $\{(x_{\alpha})_{\alpha}\in \ell^{\infty}(\kappa,\X):\ \lim_{\alpha,\mathcal{I}}\|x_{\alpha}\|= 0\}$ is its (closed, $2$-sided) ideal. We note that the technical homogeneity property on $\mathcal{I}$ is motivated by considerations in \cite{RT}. For example, the condition holds if $\mathcal{I}$ is the ideal of sets of cardinality $<\kappa$, or, even for $n=1$, if $\mathcal{I}$ is the dual ideal of an ultrafilter, which does not contain any sets of cardinality $<\kappa$ (see the proof of Lemma 3.2 in \cite{Talp}). If $\kappa=\omega$, then the above result contains the result appearing in \cite{RT}, namely  that $\ell^{\infty}(\X)\mod c_{0}(\X)$ is convex-transitive for uniformly convex-transitive $\X$.

\begin{proof}[Proof of Theorem \ref{thm: convex_ideal}]
Denote
\[N_{\mathcal{I}}=\{(x_{\alpha})_{\alpha<\kappa}\in \ell^{\infty}(\kappa,\X):\ \lim_{\alpha,\mathcal{I}}\|x_{\alpha}\|= 0\}.\] 
The fact that this is closed follows by similar argument as in the case $BC_{0}(T,\mathcal{I})$.
Note that the norm on $\ell^{\infty}(\kappa,\X)\ \mod\ N_{\mathcal{I}}$ is given by 
$\|\hat{x}\|=\inf_{I\in \mathcal{I}}\|1_{\kappa\setminus I}(\alpha)(x_{\alpha})\|$ for $\hat{x}=(x_{\alpha})+N_{\mathcal{I}}$.

Let $\phi\colon \kappa\to \kappa$ be a mapping satisfying (i). We wish to verify that 
\begin{equation}\label{eq: Iform}
\Phi\colon (x_{\alpha})_{\alpha<\kappa}+N_{\mathcal{I}}\mapsto (x_{\phi(\alpha)})_{\alpha<\kappa}+N_{\mathcal{I}}
\end{equation}
defines an isometric automorphism on $\ell^{\infty}(\kappa,\X)\ \mod\ N_{\mathcal{I}}$. Clearly $\Phi$ is linear.
Since $\phi^{-1}(I)\in \mathcal{I}$ for $I\in \mathcal{I}$, we obtain that $\Phi$ is contractive, and  
it follows from assumption (i) that $\Phi$ is really an isometry. To show that $\Phi$ is onto, we will 
find $I,J\in \mathcal{I}$ such that $\phi|_{\kappa\setminus I}$ is a bijection $\kappa\setminus I\to \kappa\setminus J$. 
One can deduce from (i) that 
\begin{equation}\label{eq: dualfilter}
\kappa\setminus \phi(\kappa\setminus I)\in \mathcal{I}\ \mathrm{if\ and\ only\ if}\ I\in \mathcal{I}.
\end{equation}
By the Hausdorff maximality principle let $\{C_{\beta}\}_{\beta}$ be a maximal increasing chain of 
subsets of $\kappa$ such that $\phi|_{C_{\beta}}$ is injective for each $\beta$. Then, by putting 
$\Gamma=\bigcup_{\beta}C_{\beta}$, we have that $\phi|_{\Gamma}$ is injective and $\phi(\Gamma)=\phi(\kappa)$. 
Since $\emptyset\in \mathcal{I}$, we obtain by \eqref{eq: dualfilter} that 
$J=\kappa\setminus \phi(\Gamma)=\kappa\setminus \phi(\kappa)\in \mathcal{I}$ and by (i) that
$I=\kappa\setminus \Gamma\in \mathcal{I}$. 

Let $\hat{x},\hat{y}\in \S_{\ell^{\infty}(\kappa,\X)\ \mod\ N_{\mathcal{I}}}$. We wish to check that 
$\hat{x}\in \overline{\conv}(\{RT\hat{y}:\ RT\})\subset \ell^{\infty}(\kappa,\X)\ \mod\ N_{\mathcal{I}}$, where $T$ ranges over isometries of the form \eqref{eq: Iform} and $R$ ranges in $\prod_{\alpha<\kappa}\mathcal{G}_{\X}$.
Towards this, let $(x_{\alpha})_{\alpha<\kappa},(y_{\alpha})_{\alpha<\kappa}\in \ell^{\infty}(\kappa,\X)$ be the corresponding representatives. Fix $z_{0}\in \S_{\X}$. Since $\X$ is uniformly convex-transitive, we obtain that 
\[(x_{\alpha})_{\alpha<\kappa}\in \overline{\conv}(\{R [(\|x_{\alpha}\|z_{0})_{\alpha<\kappa}]:\ R\})\subset \ell^{\infty}(\kappa,\X)\]
and 
\[(\|y_{\alpha}\|z_{0})_{\alpha<\kappa}\in \overline{\conv}(\{S [(y_{\alpha})_{\alpha<\kappa}]:\ S\})\subset \ell^{\infty}(\kappa,\X)\]
where $R$ and $S$ range in $\prod_{\alpha<\kappa}\mathcal{G}_{\X}$. It is easy to see that this type of isometries on $\ell^{\infty}(\kappa,\X)$ define in a natural way also isometries on $\ell^{\infty}(\kappa,\X)\ \mod\ N_{\mathcal{I}}$.

Thus, it suffices to verify that 
\[(\|x_{\alpha}\|z_{0})_{\alpha<\kappa}+N_{\mathcal{I}}\in \overline{\conv}(\{T [(\|y_{\alpha}\|z_{0})_{\alpha<\kappa}+N_{\mathcal{I}})]:\ T\}) \subset \ell^{\infty}(\kappa,\X)\ \mod\ N_{\mathcal{I}},\] 
where $T$ ranges over isometries of the type \eqref{eq: Iform}. Let $\epsilon^{\prime}>0$. 
Put 
$A=\{\alpha\in\kappa:\ \|y_{\alpha}\|>1-\epsilon^{\prime}\}$ and 
$B=\{\alpha\in \kappa:\ \|x_{\alpha}\|>\epsilon^{\prime}\}$. Note that $A,B\notin \mathcal{I}$ by the selection of $x$ and $y$.

It follows from assumption (ii) by using
isometries of the form \eqref{eq: Iform} that there is $(c_{\alpha})\in \ell^{\infty}(\kappa)$ such that 
$(c_{\alpha}z_{0}) + N_{\mathcal{I}}\in \overline{\conv}(\{T [(\|y_{\alpha}\|z_{0})_{\alpha<\kappa}+N_{\mathcal{I}})]:\ T\}) \subset \ell^{\infty}(\kappa,\X)\ \mod\ N_{\mathcal{I}}$, where $c_{\alpha}\in [1-\epsilon^{\prime},1]$ for $\alpha\in B$. 
Since $\epsilon^{\prime}$ was arbitrary, we obtain by applying suitable convex combinations
with changes of signs, that 
\[(1_{B}(\alpha)z_{0})_{\alpha<\kappa}+N_{\mathcal{I}}\in \overline{\conv}(\{T [(\|y_{\alpha}\|z_{0})_{\alpha<\kappa}+N_{\mathcal{I}})]:\ T\}) \subset \ell^{\infty}(\kappa,\X)\ \mod\ N_{\mathcal{I}},\]   
and further that
\[(\|x_{\alpha}\|z_{0})_{\alpha<\kappa}+N_{\mathcal{I}}\in \overline{\conv}(\{T [(\|y_{\alpha}\|z_{0})_{\alpha<\kappa}+N_{\mathcal{I}})]:\ T\}) \subset \ell^{\infty}(\kappa,\X)\ \mod\ N_{\mathcal{I}}.\] 
\end{proof}

The following example is an application of the above result. Let $\X$ be the Banach algebra of \emph{all} bounded functions $f\colon [0,1]\to \R$ with pointwise operations and the $\sup$ norm. 
Let $\mathcal{I}\subset \X$ be the ideal consisting of functions, whose support has cardinality strictly less than the continuum. Assuming CH would translate to the statement that their support is countable or finite, but this is not essential here.
Now, Theorem \ref{thm: convex_ideal} yields that $\X\mod \mathcal{I}$ is uniformly convex-transitive (regardless of the verity of CH). If one considers $\Y$ to be the Banach algebra of measurable functions $g\colon [0,1]\to \R$ instead, again with the $\sup$ norm, and $\mathcal{J}$ is the ideal of the functions, whose support has Lebesgue measure $0$, then $\Y\mod\mathcal{J}=L^{\infty}$ isometrically, a classical example of a convex-transitive space (\cite{Rol}).

Let us denote 
\[\ell^{\infty}_{\sigma}(\kappa,\X)=\{\{x_\alpha\}_{\alpha<\kappa}\in \ell^{\infty}(\kappa,\X):\ |\mathrm{supp}(x)|\leq \sigma\}\subset \ell^{\infty}(\kappa,\X),\]
where $\kappa,\ \sigma$ are infinite cardinals, and $\sigma\leq \kappa$.

Let $\kappa$ and $\lambda$ be cardinals of uncountable cofinality. Next, we will study the Banach algebra $\ell^{\infty}(\kappa^{+},\ell^{\infty}(\lambda^{+}))$, where the algebra operation is the natural pointwise multiplication.
We denote by
\[\mathcal{I}=\ell^{\infty}_{\kappa}(\kappa^{+},\ell^{\infty}(\lambda^{+}))+\ell^{\infty}(\kappa^{+},\ell^{\infty}_{\lambda}(\lambda^{+}))\subset \ell^{\infty}(\kappa^{+},\ell^{\infty}(\lambda^{+}))\]
depending on the fixed $\kappa,\lambda$, and we note that this is an ideal of the Banach algebra $\ell^{\infty}(\kappa^{+},\ell^{\infty}(\lambda^{+}))$. The fact that $\mathcal{I}$ is closed follows from 
the uncountable cofinality of $\kappa$ and $\lambda$.
By modifying the proof of Theorem \ref{thm: convex_ideal} we obtain the following result.
\begin{theorem}
Let $\kappa,\lambda$ and $\mathcal{I}$ be as above. Then the Banach algebra $\ell^{\infty}(\kappa^{+},\ell^{\infty}(\lambda^{+}))\mod \mathcal{I}$ is convex-transitive.
\end{theorem}
\qed

\subsection{Corona type algebras}

In what follows all Hilbert spaces are considered over the complex field.
Becerra and Rodriguez proved (\cite[Corollary 4.6]{BR_manuscripta}) that the Calkin algebra 
$\mathcal{C}(\ell^2 )=\mathcal{B}(\ell^{2})\mod \mathcal{K}(\ell^{2})$ is convex-transitive. Next we will give 
a variant of this result.

\begin{theorem}
The algebra $\mathcal{B}(\ell^{2}(\omega_{1}))\mod \mathcal{S}(\ell^{2}(\omega_{1}))$ is convex-transitive.
The algebra $\mathcal{C}(\ell^{2}(\kappa))$ is $\tau$-convex-transitive, where 
$\kappa$ is any infinite cardinal and $\tau$ is the topology of uniform convergence in separable subspaces. 
\end{theorem}

\begin{proof}
The proof is a modification of the proof of \cite[Theorem 4.5]{BR_manuscripta}, which states that
$\B_{\mathcal{B}(\ell^{2})}\subset\overline{\conv}(\mathcal{G}_{\mathcal{B}(\ell^{2})}(T))$, whenever $T\in \mathcal{B}(\ell^{2})$ satisfies
$\|T+\mathcal{K}(\ell^{2})\|_{\mathcal{B}(\ell^{2})\mod \mathcal{K}(\ell^{2})}=1$. We denote by $\{e_{\gamma}\}_{\gamma<\kappa}$ the canonical orthonormal basis of $\ell^{2}(\kappa)$. The case where $\kappa$ is countable is known from the above 
reference.

\textit{Observation 1}. Given $T\in \mathcal{B}(\ell^{2}(\kappa))$ such that $\|T+\mathcal{K}(\ell^{2}(\kappa))\|_{\mathcal{B}(\ell^{2}(\kappa))\mod \mathcal{K}(\ell^{2}(\kappa))}=1$, and a separable subspace $\Y\subset \ell^{2}(\kappa)$, there is a countable subset $\Gamma\subset \kappa$ such that 
$T(\ell^{2}(\Gamma))\subset \ell^{2}(\Gamma)\supset \Y$, and 
$\|T|_{\ell^{2}(\Gamma)} + \mathcal{K}(\ell^{2}(\Gamma))\|_{\mathcal{B}(\ell^{2}(\Gamma))\mod \mathcal{K}(\ell^{2}(\Gamma))}=1$.
Indeed, this can be obtained as follows. 
We select a countable subset $\Gamma_{0}\subset \kappa$ such that $\Y\subset \ell^{2}(\Gamma_{0})$ and 
$\|T|_{\ell^{2}(\Gamma_{0})}+\mathcal{K}(\ell^{2}(\Gamma_{0}))\|_{\mathcal{B}(\ell^{2}(\Gamma_{0}))\mod \mathcal{K}(\ell^{2}(\Gamma_{0}))}=1$. 
Then we define recursively $\Gamma_{n+1}$ to be the union of $\Gamma_{n}$ and the countable support
of $T(\ell^{2}(\Gamma_{n}))$ for $n<\omega$. Then the set $\Gamma=\bigcup_{n<\omega}\Gamma_{n}$ is as required. 
Clearly $T|_{\ell^{2}(\Gamma)}$ is positive if $T$ is such.
\textit{Observation 2.} Let $P$ be an orthogonal projection on $L^{2}$ with infinite-dimensional range. 
Then for each $n\in \N$ there are rotations $R_{1},R_{2},\ldots, R_{n}\in \mathcal{G}_{L^{2}}$ such that 
\[\|\I-\frac{1}{n}\sum_{i=1}^{n}R_{i}^{-1}PR_{i}\|\leq \frac{1}{n}.\]
Indeed, by rotating $L^{2}$ we may assume without loss of generality that the range of $P$ contains 
that of $1_{[0,1/2]}\cdot \I$. We may select the isometries $R_{i}$ in such a way that 
the range of $R_{i}^{-1}PR_{i}$ contains that of $1_{[0,1-2^{-i}]\cup [1-2^{-i-1},1]}\cdot \I$, which yields the claim. 

Next we will indicate the required changes to the proof of \cite[Theorem 4.5]{BR_manuscripta} in order to 
verify the first part of the statement. We fix $\tilde{T}\in \mathcal{B}(\ell^{2}(\omega_{1}))\mod \mathcal{S}(\ell^{2}(\omega_{1}))$ with $\|\tilde{T}\|=1$. By using the regularity of $\omega_{1}$ we may select a representative $T\in \mathcal{B}(\ell^{2}(\omega_{1}))$ having the same norm. We treat only the case where 
$T$ is positive. 

We apply recursion of length $\omega_{1}$ to construct a partitioning of 
$\omega_{1}$ into countable subsets $\Gamma_{\alpha},\ \alpha<\omega_{1},$ such that 
\[\|P_{\alpha}T|_{\ell^{2}(\Gamma_{\alpha})}+\mathcal{K}(\ell^{2}(\Gamma_{\alpha}))\|_{\mathcal{B}(\ell^{2}(\Gamma_{\alpha}))\mod \mathcal{K}(\ell^{2}(\Gamma_{\alpha}))}=1\quad \mathrm{for}\ \alpha< \omega_{1}.\]
Here $P_{\alpha}$ is the orthogonal projection onto $\ell^{2}(\Gamma_{\alpha})$. Indeed, this can be obtained by applying Observation 1. and the fact that the countable segments of $\omega_{1}$ are negligible in calculating the quotient norm $\|\cdot\|_{\mathcal{B}(\ell^{2}(\omega_{1}))\mod \mathcal{S}(\ell^{2}(\omega_{1}))}$.

Let us consider self-adjoint idempotents $\pi_{\alpha}\in \mathcal{B}(\ell^{2}(\Gamma_{\alpha}))$, more precisely, 
orthogonal projections to infinite-dimensional subspaces provided by the spectral theorem, similarly as in the proof of \cite[Theorem 4.5]{BR_manuscripta}, such that $\|\pi_{\alpha}-\pi_{\alpha} P_{\alpha}T|_{\ell^{2}(\Gamma_{\alpha})}\|_{\mathcal{B}(\ell^{2}(\Gamma_{\alpha}))}<\epsilon$ for $\alpha<\omega_{1}$. 
Denote by $\pi_{\alpha}^{\bot}$ the coprojection of $\pi_{\alpha}$.
By considering the averages of $P_{\alpha}T$ and $(\I-2\pi_{\alpha}^{\bot})P_{\alpha}T$ we observe that 
\[\pi_{\alpha}\in \overline{\left\{\frac12 (P_{\alpha}T|_{\ell^{2}(\Gamma_{\alpha})}+RP_{\alpha}T|_{\ell^{2}(\Gamma_{\alpha})}):\ R\in 
\mathcal{G}_{\mathcal{B}(\ell^{2}(\Gamma_{\alpha}))}\right\}}.\] 

By Observation 2. we have that 
\[\prod_{\alpha<\omega_{1}} \I_{\ell^{2}(\Gamma_{\alpha})}\in \overline{\conv}(\{\ \{R_{\alpha}^{(1)}P_{\alpha}T R_{\alpha}^{(2)}\}_{\alpha}:\ 
\{R_{\alpha}^{(1)}\}_{\alpha},\{R_{\alpha}^{(2)}\}_{\alpha}\in \prod_{\alpha}\mathcal{G}_{\mathcal{B}(\ell^{2}(\Gamma_{\alpha}))}\}.\]
By using the Russo-Dye Theorem we obtain that 
$\B_{\mathcal{B}(\ell^{2}(\omega_{1}))}=\overline{\conv}(\mathcal{G}_{\ell^{2}(\omega_{1})})$. Thus 
$\B_{\mathcal{B}(\ell^{2}(\omega_{1}))}=\overline{\conv}(\mathcal{G}_{\mathcal{B}(\ell^{2}(\omega_{1}))}(T))$. 

It is clear that all the isometries of $\mathcal{B}(\ell^{2}(\omega_{1}))$ applied above are realized as multiplications
$x\mapsto axb$. Thus these isometries preserve the multiplicative structure of the space. This yields that 
all $2$-sided ideals are preserved by these isometries. In particular, this holds for the ideal 
$\mathcal{S}(\ell^{2}(\omega_{1})$, and thus the isometries considered induce isometries on 
$\mathcal{B}(\ell^{2}(\omega_{1}))\mod \mathcal{S}(\ell^{2}(\omega_{1}))$. This justifies the first part of the statement. 

To verify the second part of the statement, fix $T,T_{1}, U_{1},\ldots U_{n}\in \mathcal{B}(\ell^{2}(\kappa))$ with 
$\|T_{1}\|_{\mathcal{B}(\ell^{2}(\kappa))}\leq \|T+\mathcal{K}(\ell^{2}(\kappa))\|_{\mathcal{B}(\ell^{2}(\kappa))\mod \mathcal{K}(\ell^{2}(\kappa))}=1$, a separable subspace 
$\Y\subset \ell^{2}(\kappa)$ and $\epsilon>0$ such that 
$\|T_{1}|_{\Y} - U_{1}|_{\Y}\|_{\mathcal{B}(\Y,\ell^{2}(\kappa))},\ldots,\|T_{1}|_{\Y} - U_{n}|_{\Y}\|_{\mathcal{B}(\Y,\ell^{2}(\kappa))}<\epsilon$.
We have just fixed a $\tau$-open neighbourhood of $T_{1}$, see \eqref{eq: tauneigh}.

By using the argument of Observation 1. we get that there is a countable subset $\Gamma\subset \kappa$
such that 
\begin{enumerate}
\item[(i)]{$\Y\subset \ell^{2}(\Gamma)$.}
\item[(ii)]{$T(\ell^{2}(\Gamma)),T_{1}(\ell^{2}(\Gamma)),U_{1}(\ell^{2}(\Gamma)),\ldots,U_{n}(\ell^{2}(\Gamma)) \subset \ell^{2}(\Gamma)$.}
\item[(iii)]{$\|T|_{\ell^{2}(\Gamma)}+\mathcal{K}(\ell^{2}(\Gamma))\|_{\mathcal{B}(\ell^{2}(\Gamma))\mod \mathcal{K}(\ell^{2}(\Gamma))}=1$.}
\end{enumerate}

Note that our considerations reduce to the space $\ell^{2}(\Gamma)$ by applying the average of the isometries
$\I$ and $1_{\ell^{2}(\Gamma)}\cdot \I - 1_{\ell^{2}(\kappa\setminus \Gamma)}\cdot \I$. Note that multiplication with these isometries 
gives isometries on $\mathcal{B}(\ell^{2}(\kappa))$ and also preserve the ideal $\mathcal{K}(\ell^{2}(\kappa))$.

Now we may apply \cite[Theorem 4.5]{BR_manuscripta}, which yields that 
\[T_{1}|_{\ell^{2}(\Gamma)}\in \overline{\conv}(\mathcal{G}_{\mathcal{B}(\ell^{2}(\Gamma))}(T|_{\ell^{2}(\Gamma)}))
\subset \mathcal{B}(\ell^{2}(\Gamma)).\]
This means that the set of mappings $\conv(\mathcal{G}_{\mathcal{B}(\ell^{2}(\Gamma))}(T|_{\mathcal{B}(\ell^{2}(\Gamma))}))$,
embedded in $\mathcal{B}(\ell^{2}(\kappa))$ in a natural way, intersects the $\tau$-open neighbourhood 
\begin{equation}\label{eq: tauneigh}
\{U\in \mathcal{B}(\ell^{2}(\kappa)):\ \|U|_{\Y} - U_{1}|_{\Y}\|_{\mathcal{B}(\Y,\ell^{2}(\kappa))},\ldots,
\|U|_{\Y} - U_{n}|_{\Y}\|_{\mathcal{B}(\Y,\ell^{2}(\kappa))}<\epsilon\}
\end{equation}
of $T_{1}$. Since $T_{1}$ and the neighbourhood were arbitrary, we have that 
\[\B_{\mathcal{B}(\ell^{2}(\kappa))}=\overline{\conv}^{\tau}(\mathcal{G}_{\mathcal{B}(\ell^{2}(\kappa))}(T)).\]

The ideal $\mathcal{K}(\ell^{2}(\kappa))$ is invariant under the rotations applied, as mentioned in the proof of 
\cite[Corollary 4.6]{BR_manuscripta} and thus we have the claim.
\end{proof}

Next we will study the Banach algebra $\mathcal{B}(L^{\infty})$, where the composition of maps is the algebra operation. 

We denote by $\mathcal{I}$ a (non-unital) left ideal of $\mathcal{B}(L^{\infty})$ consisting of elements $T$
such that $\lim_{t\to 1}\|T\circ 1_{[t,1]}\cdot \I\|=0$. Let $\mathfrak{M}$ be the Stone space of $\mathrm{Bor}([0,1])\mod \ker(m)$ and we regard $C(\mathfrak{M})=L^{\infty}$ isometrically. Recall that $C(K)^{\ast}=\mathrm{M}(K)$ for any compact $K$ by the Riesz theorem. Let us consider $M=\bigwedge_{t<1}\neg [0,t]\subset \mathfrak{M}$ under natural identifications. This set consists of continuum many ultrafilters on $\mathrm{Bor}([0,1])\mod \ker(m)$.

We define a subalgebra $\mathcal{A}\subset \mathcal{B}(L^{\infty})$ by 
\[\mathcal{A}=\mathcal{B}(L^{\infty})\cap \mathcal{I}^{\bot\bot}\subset \mathcal{B}(L^{\infty})^{\ast\ast}.\]

It is easy to see that $\mathcal{A}$ contains the space $\mathcal{B}_{\ast}(L^\infty)$ of the dual (i.e. $\omega^{\ast}$-$\omega^{\ast}$-continuous) operators, e.g. $\I_{L^{\infty}}$. An essential point here is that $\mathcal{A}$ does not contain any operators $T\in \mathcal{B}(L^{\infty}),\ T\neq 0,$ such that $T\circ 1_{[0,t]}\cdot \I=0$ 
for all $t<1$. Simple examples of such operators $T$ on $C(\mathfrak{M})$ are the finite-rank operators $f\mapsto \sum_{i} x_{i}^{\ast}(f)y_{i}$, where $y_{i}\in C(\mathfrak{M})$ and $x_{i}^{\ast}$ are the point evaluations $\delta_{k_{i}}$ for points $k_{i}$ in the compact set $M$. In fact, for all $T\in \mathcal{B}(L^{\infty})$ and all $x^{\ast}\in C(\mathfrak{M})^{\ast}$ it holds that $x^{\ast}\circ T \in \mathrm{M}(\mathfrak{M})$ is supported outside of $M$ (see Proposition \ref{prop: Mideal}). We denote $\mathcal{I}_{\ast}=\mathcal{I}\cap \mathcal{B}_{\ast}(L^\infty)$.

Observe that fixing a positive functional $f \in \S_{c_{0}^{\bot}}\subset (\ell^{\infty})^{\ast}$ and writing
\[F(T)=f(\{\E(T(1_{[n,n+1]})\ |\ 1_{[n,n+1]})\}_{n\in\N}),\quad T\in \mathcal{B}(L^{\infty})\] 
defines a $\mathrm{SOT}$-continuous functional in $\S_{\mathcal{B}(L^{\infty})^{\ast}}$ such that $F(\mathcal{I})=\{0\}$ and $F(\I)=1$.

\begin{theorem}\label{thm: Calkin}
The algebras $\mathcal{A}\mod \mathcal{I}$ and $\mathcal{B}_{\ast}(L^{\infty})\mod \mathcal{I}_{\ast}$ are $\mathrm{SOT}$-convex-transitive. 
\end{theorem}

We will apply inverse systems of suitable approximations in the proof, which closely resembles a rather usual application of Markov chain discretization in Ergodic theory. Also, one recognizes a martingale condition in the formula \eqref{eq: martingale}. On the other hand, the philosophy of the construction is very similar to the one discussed in connection with the continuous functions on the Cantor set in Section \ref{sect: tree}.

\begin{proof}[Proof of Theorem \ref{thm: Calkin}]
The proofs for the claims are similar, except that the case of $\mathcal{B}_{\ast}(L^{\infty})\mod \mathcal{I}_{\ast}$ is a bit easier, so we will restrict our attention to the algebra $\mathcal{A}\mod \mathcal{I}$.
 
We will work with \emph{finite} decompositions coming from $\mathrm{Bor}([0,1])\mod \ker(m)$, written $\rho$, or so, 
and the square $\rho\times \rho$ is denoted by $\tilde{\rho}$. Let $\Lambda$ be the lattice of all such decompositions $\tilde{\rho}$. It has the operation $\tilde{\rho}_{1}\vee\tilde{\rho}_{2}$ of taking the unique common refinement of the decompositions, and $\tilde{\rho}_{1}\wedge\tilde{\rho}_{2}$ is the finest decomposition such that both $\tilde{\rho}_{1}$
and $\tilde{\rho}_{2}$ refine it. This lattice is bounded below by the element $[0,1]^{2}$.

For all $T\in \mathcal{B}(L^{\infty}),\ x\in L^{\infty}$ we define an inverse system $\{M_{\tilde{\rho}}\}_{\tilde{\rho}\in \Lambda}$
of square matrices of reals, essentially of mappings $\{1,\ldots,n\}^2 \to \R$ under suitable identifications, but it is more convenient write them as $M_{\tilde{\rho}}=M_{\tilde{\rho}}(T)[x]$. These are defined as follows: given $\tilde{\rho}\in \Lambda$ we put 
$M_{\tilde{\rho}}=[c_{A,B}]_{(A,B)\in \tilde{\rho}}$, where  
\[c_{A,B}(T,x)=\E(T(\E(x\ |\ B)1_{B})\ |\ A),\quad x,1_{B}\in L^{\infty}.\]
The corresponding finite-rank approximation of $T$ and $x$ are 
\[T_{\tilde{\rho}}[x]=\sum_{A\in \rho} (\sum_{B\in \rho}c_{A,B}(T,x))1_{A},\ A_{\rho}(x)=\sum_{B\in \rho}\E(x\ |\ B) 1_{B}, \]
and it is easy to verify that $T_{\tilde{\rho}}(x)=A_{\rho}TA_{\rho}(x)$. Clearly $A_{\rho}$ is a contractive projection.

The binding maps are given as follows: if $\tilde{\rho}_{1}\leq \tilde{\rho}_{2}$
then $\phi_{\tilde{\rho}_{1}}^{\tilde{\rho}_{2}}\colon M_{\tilde{\rho}_{2}}\mapsto M_{\tilde{\rho}_{1}}$ 
is given by 
\[[\phi_{\tilde{\rho}_{1}}^{\tilde{\rho}_{2}}(M_{\tilde{\rho_{2}}})]_{(A,B)\in \tilde{\rho}_{1}}(T)[x]= 
[\ \E(T_{\tilde{\rho}_{2}}(\E(x\ |\ B)1_{B})\ |\ A)\ ]_{(A,B)\in \tilde{\rho_{1}}}
.\]
Observe that the corresponding operators then satisfy
\begin{equation}\label{eq: martingale}
1_{A}\cdot T_{\tilde{\rho}_{2}}[1_{B}x]=1_{A}\cdot T_{\tilde{\rho}_{1}}[1_{B}x],\quad \mathrm{for}\ (A,B)\in \tilde{\rho}_{1}.
\end{equation}

Since the simple functions are dense in $L^{\infty}$, we observe that the net $(A_{\rho}(x))_{\rho}$ converges to $x$
in the norm. Similarly, the net $(A_{\rho}(Tx))_{\rho}$ converges to $T(x)$.
Given $x\in L^{\infty}$ and $\tilde{\rho}_{1}$, we obtain according to the compatibility condition
\eqref{eq: martingale} that $A_{\rho_{1}}(T-T_{\tilde{\rho}_{2}})A_{\rho_{1}}x=A_{\rho_{1}}(T-A_{\rho_{2}}T)A_{\rho_{1}}x$ for 
any $\tilde{\rho}_{1}\leq \tilde{\rho}_{2}$. Now, we we deduce from the convergence of the above nets that 
$\lim_{\tilde{\rho}_{2}}\|A_{\rho_{1}}(T-T_{\tilde{\rho}_{2}})x\|=0$ for any $x$ and $\rho_{1}$. By using the fact that 
the simple functions are dense in $L^{\infty}$ we conclude that 
\begin{equation}\label{eq: taulim}
\mathrm{SOT}-\lim_{\tilde{\rho}}T_{\tilde{\rho}}=T.
\end{equation}

We claim that for each $T\in \mathcal{A}$ it holds that 
$\lim_{t \to 1}\|T\circ 1_{[s,t]}\cdot \I\|=\|T\circ 1_{[s,1]}\cdot \I\|$ for $0<s<1$. Indeed, assume to the contrary that 
\begin{equation}\label{eq: limT}
\lim_{t \to 1}\|T\circ 1_{[s,t]}\cdot \I\|=\|T\circ 1_{[s,1]}\cdot \I\| - \delta
\end{equation}
for a given $s$ and $\delta>0$. Let $f\in \S_{L^{\infty}}$ be such that 
$\|T\circ 1_{[s,1]}\cdot \I (f)\|> \|T\circ 1_{[s,1]}\cdot \I\|-\delta/2$. 
Then fix $F\in \S_{\mathcal{A}^{\ast}}$ such that $F(T\circ 1_{[s,1]}\cdot \I (f))> \|T\circ 1_{[s,1]}\cdot \I (f)\| - \delta/2$. Thus $F(T\circ 1_{[s,1]}\cdot \I (f))> \|T\circ 1_{[s,1]}\cdot \I\|-\delta$. Recall that 
$F\circ T$ is supported outside of $M$, see Proposition \ref{prop: Mideal}. Since $t\mapsto F\circ T (1_{[s,t]}\cdot f)$ is not continuous at $1$ according to \eqref{eq: limT}, we have a contradiction. 

Fix $\widehat{T},\widehat{S}\in \S_{\mathcal{A}\mod \mathcal{I}}$, and we will pick their representatives $T,S\in \mathcal{B}(L^{\infty})$, respectively. To verify the $\mathrm{SOT}$-convex-transitivity, our strategy is to apply rotations
$R$ on $\mathcal{B}(L^{\infty})$ fixing the subspace $\mathcal{I}$ as a subset (and thus also $\mathcal{A}=\mathcal{B}(L^{\infty})\cap \mathcal{I}^{\bot\bot}$) such that $T\in \overline{\conv}^{\mathrm{SOT}}(RS:\ R)$. We may assume without loss of generality, possibly by applying approximation later, that $\|T\|=\|S\|=1$. Actually, according to \eqref{eq: taulim} we are only required to show that 
$T_{\tilde{\rho}}\in \overline{\conv}^{\mathrm{SOT}}(RS: R)$ for each $\tilde{\rho}\in \Lambda$.

We claim that for each $\epsilon>0$ there exists a sequence $(A_{i})$ of successive open intervals $A_{i}\subset [0,1]$ 
such that $\|S \circ 1_{A_{i}}\cdot \I\|>1-\epsilon$. Indeed, the function 
$t\mapsto  \|S \circ 1_{[0,t]}\cdot \I\|$ has the value $1$ at $t=1$, where the function is continuous. 
Therefore we may pick $s_{1}<1$ such that $\|S \circ 1_{[0,s_{1}]}\cdot \I\|>1-\epsilon$. Put $A_{1}=(0,s_{1})$.
It follows from the definition of the $\|\cdot\|_{\mathcal{A}\mod \mathcal{I}}$ norm that $\|S \circ 1_{[s_{1},1]}\cdot \I\|=1$.
We apply the above continuity observation again to obtain a real $s_{2}\in (s_{1},1)$ such that 
$\|S \circ 1_{[s_{1},s_{2}]}\cdot \I\|>1-\epsilon$. Then we put $A_{2}=(s_{1},s_{2})$. 
We proceed recursively in this fashion to get the required sequence of intervals $A_{i+1}=(s_{i},s_{i+1})$.

Since the sets $A_{i}$ are disjoint, we have $\|S\circ 1_{A_{i}}\cdot \I\|\geq 1-\epsilon$ for all $i$. For each $i$ let 
\[B_{i}=\bigvee_{y}[\{t\in [0,1]:\ |y(t)|\geq 1-\epsilon\}]\quad (\mathrm{mod}\ \ker(m))\]
where $y$ ranges in elements of $L^{\infty}$ having the form $y=S(1_{A_{i}}x),\ x\in L^{\infty}$. Observe that 
\begin{equation}\label{eq: 1B}
\|1_{B_{j}}\cdot S \circ 1_{A_{i}}\cdot \I\|\leq \epsilon\ \mathrm{for}\ i\neq j. 
\end{equation}

Fix a decomposition $\tilde{\rho}\in \Lambda$. Let us enumerate $\rho$ in the following way: $\rho\ni C_{i}\subset [0,1],\ 1\leq i \leq k,$ such that $\mathrm{ess\ sup}\ C_{i}=1$ and $\rho\ni C_{i}\subset [0,1],\ k+1\leq i\leq m,$ such that 
$\mathrm{ess\ sup}\ C_{i}<1$. We rotate $\mathcal{B}(L^{\infty})$ by permuting the measure algebra 
$\mathrm{Bor}([0,1])\mod \ker(m)$ in such a way that we may identify $C_{k+1}=A_{1},\ldots, C_{m}=A_{m-k}$ 
and $C_{i}=\bigvee_{n\in \N}A_{i+nm-k}$ for $1\leq i\leq m$. 
Here the careful selection of the correspondences is required to maintain $\mathcal{A}$ and $\mathcal{I}$. 
On the other hand, since these are only \emph{left} ideals, such restrictions do not apply on mapping the sets $B_{i}$.
We collect and enumerate the sets $B_{i}$ according to the same rule by which we treated the sets $A_{i}$ 
such that \eqref{eq: 1B} holds.

By applying obvious rotations $R$ induced by $L^{\infty}=\ell^{\infty}(L^{\infty})$, which is a uniformly convex-transitive space, we obtain that there is $S_{0}\in \overline{\conv}(\{RS:\ R\})$ such that $1_{(2^{i-1},2^{i})}(t)S_{0}1_{(2^{j-1},2^{j})}x=\delta_{i,j}$ a.e. Indeed, the set $[0,1]\setminus \bigvee B_{i}$ can be neglected by applying an averaging sliding hump argument with similar idea as in \eqref{eq: phiit} (or in the argument of 
\cite[Thm. 3.4]{RT} involving sets $\Delta_{n}$) and $\epsilon$ in \eqref{eq: 1B} can be taken to $0$, since we take 
the closure.

By applying the above sets we aim to find 
$S_{1}\in \overline{\conv}(\mathcal{G}_{\mathcal{A}\mod\mathcal{I}}(S_{0}))$ such that 
$\|A_{\rho}(T - S_{1})A_{\rho}\|<\epsilon$.
We claim that $\{G_{\tilde{\rho}}:\ G\in \overline{\conv}(\mathcal{G}_{\mathcal{A}\mod\mathcal{I}}(S_{0}))\}$ contains all norm-$1$ finite-rank operators of the form 
\begin{equation}\label{eq: fsum}
f\mapsto \sum_{i} \E(f\ |\ C_{i}) 1_{E_{i}}, 
\end{equation}
where the sets $E_{i}$ are disjoint and formed by taking unions of sets $C_{i}$. We do not assume that $\bigvee_{i} E_{i}=[0,1]$. Indeed, this is obtained by applying obvious rotations on $L^{\infty}$ and by applying the 
abovementioned averaging sliding hump argument. 

Let $M_{\tilde{\rho}}$$\in \R^{n\times n}$ be the matrix corresponding to $T_{\tilde{\rho}}$. The absolute values of reals on each row sum up to $1$ at most. This matrix can be thought of as a linear operator on the finite-dimensional space of 
simple functions determined by the partition $\rho$. We must show that $M_{\tilde{\rho}}$ can be approximated by taking averages of matrix representations of operators of the form \eqref{eq: fsum}. In doing this, we may assume without loss of generality that each entry of $M_{\tilde{\rho}}$ is of the form $\pm k/n$ with $k\in \{0,1,\ldots, n\}$, since $n$ can be taken to be large. 

We form a '$3$-dimensional matrix' $N\colon \{1,\ldots, n\}^3 \to \R$ as follows: For each $i$ we put 
$N(i,1,m)=1$ for $m\leq |k_{i,1}|$ and $N(i,1,m)=0$ for $m>|k_{i,1}|$ where $k_{i,1}/n=M_{\tilde{\rho}}(i,1)$. In general, 
we define $N(i,j,m)=1$ for $\sum_{l=1}^{j-1}|k_{i,l}|<m\leq \sum_{l=1}^{j}|k_{i,l}|$ and $0$ for other values of $m$. 
Then the approximation for $T_{\tilde{\rho}}$ is obtained as the average
\[x\mapsto \frac{1}{n}\sum_{m=1}^{n} \sum_{i,j}N(i,j,m) \E(x\ | C_{j}) 1_{C_{i}},\quad x\in L^{\infty}\] 
of operators, which can be realized in the form \eqref{eq: fsum}.

We conclude that $T\in \overline{\conv}^{\mathrm{SOT}}(RS:\ R)$. Since the rotations $R$ applied
preserve $\mathcal{I}$ and $\mathcal{A}$, this means that $\mathcal{A}\mod \mathcal{I}$ is $\mathrm{SOT}$-convex-transitive.
\end{proof}

We do not know if the Calkin algebra $\mathcal{C}(L^{\infty})=\mathcal{B}(L^{\infty})\mod \mathcal{K}(L^{\infty})$ is ($\mathrm{SOT}$-)convex-transitive. We note that the above ideals $\mathcal{A}$ and $\mathcal{I}$
are in fact $u$-ideals, as the following result shows. 

\begin{theorem}\label{thm: Mideal}
Let $\X$ be a unital Banach algebra with a sequence of idempotents $\pi_{i},\ i\in \N,$
such that $\|\pi_{i}\|=\|1_{\X}-2\pi_{i}\|=1$ for $i$ and $\pi_{i}\pi_{j}=\pi_{j}\pi_{i}=\pi_{j}$ for $j>i$.
Consider the left ideal
\[\mathcal{I}=\{x\in \X:\ \lim_{i\to \infty}\|x\pi_{i}\|=0\}\subset \X.\]
Then $\mathcal{I}$ is a $u$-ideal.
\end{theorem}
\begin{proof}
We will define an isometric reflection projection $P\colon \X^{\ast}\to \mathcal{I}^{\bot}$. 
Then $(\I-P)^{\ast}\colon \X^{\ast\ast}\to \mathcal{I}^{\bot\bot}$ is such as well and we have the claim.

Let $\mathcal{U}$ be a non-principal ultrafilter on $\N$. For all $x^{\ast}\in \X^{\ast}$ and $x\in \X$
we define $P(x^{\ast})[x]=\lim_{i,\mathcal{U}}x^{\ast}(x\pi_{i})$. Clearly $P$ is a bounded linear operator.
Observe that 
\[P(x^{\ast})[x\pi_{n}]=\lim_{i,\mathcal{U}}x^{\ast}(x\pi_{n}\pi_{i})=\lim_{i,\mathcal{U}}x^{\ast}(x\pi_{i})
=P(x^{\ast})[x],\]
for $x\in\X$. Thus $P^{2}(x^{\ast})[x]=\lim_{n,\mathcal{U}}P(x^{\ast})(x\pi_{n})=P(x^{\ast})[x]$, so that $P$ is a projection.

Note that 
\[|(\I-2P)(x^{\ast})[x]|=|\lim_{i,\mathcal{U}}x^{\ast}(x)-x^{\ast}(x2\pi_{i})|
=\lim_{i,\mathcal{U}}|x^{\ast}(x(1_{\X}-2\pi_{i}))|\leq \|x^{\ast}\|\ \|x\|,\]
so that $\|\I-2P\|\leq 1$. The last condition characterizes isometric reflection projections among projections $P$.

Let us verify that $\mathcal{I}^{\bot}$ is the image of $P$. Note that $x(1_{\X}-\pi_{i})\in \mathcal{I}$ for all
$x\in \X$. Thus $x^{\ast}\in \mathcal{I}^{\bot}$ implies that $x^{\ast}(x(1_{\X}-\pi_{i}))=0$ holds for all $x\in \X$ and 
$i\in \N$. This yields $x^{\ast}=P(x^{\ast})$, so that $\mathcal{I}^{\bot}\subset \mathrm{Im}(P)$. 

To check the converse inclusion, pick $x^{\ast}\in \X^{\ast}\setminus \mathcal{I}^{\bot}$. 
Then there is $y\in \mathcal{I}$ such that $x^{\ast}(y)=\delta>0$. However, $P(x^{\ast})[y]=\lim_{n,\mathcal{U}} x^{\ast}(y\pi_{n})=0$ and hence $x^{\ast}\neq P(x^{\ast})$. Thus $\mathrm{Im}(P)\subset \mathcal{I}^{\bot}$.
\end{proof}

\begin{proposition}\label{prop: Mideal}
The left ideal $\mathcal{A}=\mathcal{B}(L^{\infty})\cap \mathcal{I}^{\bot\bot}\subset \mathcal{B}(L^{\infty})$
appearing in Theorem \ref{thm: Calkin} satisfies that for all $T\in \mathcal{A}$ and all $x^{\ast}\in C(\mathfrak{M})^{\ast}$ it holds that $x^{\ast}\circ T \in \mathrm{M}(\mathfrak{M})$ is supported outside of $M$.
\end{proposition}
\begin{proof}
We will apply the proof of the previous theorem. Let $P\colon \mathcal{B}(L^{\infty})^{\ast}\to \mathcal{I}^{\bot}$,
$P(x^*)[T]=\lim_{n,\mathcal{U}}x^{\ast}[T\circ 1_{[1-2^{-n},1]}\cdot \I ]$. Suppose that $x_{0}^* \in (L^{\infty})^* $
and $T_{0}\in \mathcal{A}$. 

If $\mu =x_{0}^{\ast}\circ T_{0} \in C(\mathfrak{M})^{\ast}=\mathrm{M}(\mathfrak{M})$ has a strictly positive variation on $M$ then we may select a Borel set $A\subset M$ such that $\mu(A)=\delta\neq 0$. Define $F\in\mathcal{B}(C(\mathfrak{M}))^{\ast}$ by setting $F(T)=\nu_{T}(A)$ where $\nu_{T} \in \mathrm{M}(\mathfrak{M})$ is given by $\nu_{T}=x_{0}^{\ast}\circ T\in \mathrm{M}(\mathfrak{M})$. Observe that $P(F)=F$, so that $F\in \mathcal{I}^{\bot}$. 
Now, $F(T_{0})=\delta$. Thus $T_{0}\notin \mathcal{I}^{\bot\bot}$.
\end{proof}

\subsection{Cabello's transitive abstract M-space}

Recall that the non-trivial ultraproduct of countably many almost transitive spaces is transitive, 
see e.g. \cite{BR2} for discussion. In \cite{Ca4} Cabello studied a particular case of this, namely, an abstract $\mathrm{M}$-space, which is defined as follows:
Putting 
\[\Y=\bigoplus_{i=1}^{\infty}L^{i}\quad (\mathrm{direct\ sums\ in}\ \ell^{\infty}-\mathrm{sense})\]
and $N_{\mathcal{U}}=\{(y_{i})\in \Y:\ \lim_{i,\mathcal{U}}\|y_{i}\|=0\}$, where $\mathcal{U}$ is a non-principal ultrafilter on $\N$, yields a transitive quotient space $\mathcal{M}=\Y\mod N_{\mathcal{U}}$. 

The space $\mathcal{M}$ can be endowed with the Banach lattice structure by defining 
$x\vee y = [ \{\max(x_{i},y_{i})\}_{i\in\N} ]$ and $x\wedge  y = [ \{\min(x_{i},y_{i})\}_{i\in\N} ]$
for $x,y\in \mathcal{M}$ where the maximum and the minimum are taken pointwise almost everywhere in $L^{i}$
for $i\in \N$. Let us write $|x|_{\mathcal{M}}=x\vee -x$ for $x\in \mathcal{M}$.
In fact this construction produces an abstract $\mathrm{M}$-space as observed in \cite{Ca4} and this was later exploited in \cite{Ca6} in constructing almost transitive spaces of the type $C(K,\X)$.
By substituting the spaces $L^{i}$ with $L^{\infty}$ we obtain an ultrapower of $L^{\infty}$ and we will denote it by
$(L^{\infty})^{\mathcal{U}}$. Let $J\colon (L^{\infty})^{\mathcal{U}} \to \mathcal{M}$ be the canonical identification 
$[(x_{i})]_{(L^{\infty})^{\mathcal{U}}} \mapsto [(x_{i})]_{\mathcal{M}}$, which is non-expansive but not injective or onto.

Next we will make a digression and consider $\mathcal{M}$ as a special Banach module. We define $\cdot$ to be a bilinear operation $(L^{\infty})^{\mathcal{U}}\times \mathcal{M}\to \mathcal{M}$ such that 
\begin{enumerate}
\item[(a)]{$(\mathcal{M},\cdot)$ is a unital commutative Banach $(L^{\infty})^{\mathcal{U}}$-bimodule.}
\item[(b)]{$(ab)\cdot 1_{\mathcal{M}}=a\cdot J(b)$ for $a,b\in (L^{\infty})^{\mathcal{U}}$.}
\item[(c)]{$\|a\cdot x\|\leq \|a\|\ \|x\|$ for $a\in (L^{\infty})^{\mathcal{U}},\ x\in \mathcal{M}$.}
\item[(d)]{$|a\cdot x|_{\mathcal{M}}=|a|_{(L^{\infty})^{\mathcal{U}}} \cdot |x|_{\mathcal{M}}$ for $a\in (L^{\infty})^{\mathcal{U}},\ x\in \mathcal{M}$.}
\item[(e)]{$|a\cdot x|_{\mathcal{M}}=0$ if and only if $|J(a)|_{\mathcal{M}} \wedge |x|_{\mathcal{M}}=0$ for $a\in (L^{\infty})^{\mathcal{U}},\ x\in \mathcal{M}$.}
\end{enumerate}

Indeed, given representatives $(a_{i})\in \ell^{\infty}(L^{\infty})$ and $(x_{i})\in \bigoplus L^{i}$ of 
$a\in (L^{\infty})^{\mathcal{U}}$ and $x\in \mathcal{M}$, respectively, we define $a\cdot x$ as 
$[(a_{i}x_{i})]\in \mathcal{M}$, and it is straight-forward to check that the operation $\cdot$ is well defined.
On the other hand, this type of operation $ab\mapsto [(a_{i}b_{i})]$ applied on $(L^{\infty})^{\mathcal{U}}$ makes it a unital commutative Banach algebra.

The conditions (a)-(d) are easily verified, and next we will check (e). Indeed, fix representatives $(a_{i})\in \ell^{\infty}(L^{\infty})$ and $(x_{i})\in \bigoplus L^{i}$ of $a\in (L^{\infty})^{\mathcal{U}}$ and $x\in \mathcal{M}$, respectively.
Note that $|a\cdot x|_{\mathcal{M}}=0$ if and only if $\lim_{i,\mathcal{U}}\|\ |a_{i}x_{i}|\ \|_{L^{i}}=0$
and then $\lim_{i,\mathcal{U}}\|(|a_{i}|\wedge |x_{i}|)^{2} \|_{L^{i}}=0$, which implies 
$\lim_{i,\mathcal{U}}\|\ |a_{i}|\wedge |x_{i}|\ \|_{L^{i}}=0$. In the opposite direction, we may assume that the representatives are chosen in such a way that $|a_{i}|\wedge |x_{i}|=0$ for $i$, in which case it is clear that $a\cdot x=0$.

Define a space $\mathcal{M}(\X)$ by using Bochner spaces $L^{i}(\X)$ in place of $L^{i}$ in the construction of 
$\mathcal{M}$. Thus, this is a natural 'vector-valued version' of $\mathcal{M}$.
\begin{theorem}
If $\X$ is almost transitive, then $\mathcal{M}(\X)$ is transitive. If $\X$ is uniformly convex-transitive, then
$\mathcal{M}(\X)$ and $\X^{\mathcal{U}}$ are uniformly convex-transitive. In particular, $(L^{\infty})^{\mathcal{U}}$
is uniformly convex-transitive.  
\end{theorem}
\begin{proof}
The argument follows the considerations in \cite{GJK}, \cite{Ca4} and \cite{RT}. Namely, the Bochner spaces 
$L^{p}(L^{r})$ are almost transitive for $p,r<\infty$ (see e.g. \cite{GJK}), and it is well known that the ultraproduct of such spaces is transitive (see e.g. \cite{Ca4}). 

On the other hand, one can easily check that in the space
\[\Y=\bigoplus_{n\in \N}L^{n}(\X),\quad \mathrm{(summation\ in}\ \ell^{\infty}\mathrm{-sense)},\]
given two elements $x=(x_{n})_{n\in\N}$ and $y=(y_{n})_{n\in\N}$ such that $\|x_{n}\|_{L^{n}(\X)}=\|y_{n}\|_{L^{n}(\X)}$ for $n\in\N$,
we have $x\in\overline{\conv}(R(y):\ R)$ with rotations $R=\Pi R_{n}$, where $R_{n}\in \mathcal{G}_{L^{n}(\X)}$. Here we apply the fact that the spaces $L^{n}(\X)$ are 'equi-convex-transitive', that is, 
they have a common finite upper bound $C_{\epsilon}$ for constants of uniform convex-transitivity depending on $\epsilon$.

Actually, this follows from the fact that the spaces $L^{p}(\X),\ p<\infty,$ have the same constant $K(\epsilon)$ as $\X$ itself. Indeed, the convex-transitivity of $L^{p}(\X)$ can be established by applying an approximation with an outer rotation of $L^{p}$ and convex combinations of inner rotations of $\X$ via simple functions, exploiting similar ideas as in \cite{GJK} and \cite{RT}. Since the simple functions are dense, and we applied approximation with \emph{one} outer rotation, their use will not affect the value of $K(\epsilon)$. 

The argument is completed by picking suitable representatives for $\hat{x},\hat{y}\in \S_{\mathcal{M}(\X)}$, as above. The existence of such representatives follows by inspecting the properties of the 
ultralimit, since $\lim_{n,\mathcal{U}}\|x_{n}\|=\lim_{n,\mathcal{U}}\|y_{n}\|=1$. Then we apply the above approximation of $x$. The rotations $R$ used above fix $N_{\mathcal{U}}$, so that we have the claim.

The statement regarding $\X^{\mathcal{U}}$ follows by similar reasoning.
\end{proof}

The above argument yields that if $(\Y_{n})$ is a sequence of Banach spaces that is equi-convex-transitive in the above 
sense, then the corresponding ultraproduct is uniformly convex-transitive as well. 

\section{Final remarks}

We do not know what is the relationship (if any) between the following two conditions of a Banach space 
$\X$:
\begin{enumerate}
\item[(i)]{$(\X,\|\cdot\|)$ has a \emph{maximal} norm (\cite{Wo}), that is, for any equivalent norm $|||\cdot|||$ on $\X$ the condition $\mathcal{G}_{(\X,\|\cdot\|)}\subset \mathcal{G}_{(\X,|||\cdot |||)}$ implies that $\mathcal{G}_{(\X,\|\cdot\|)}=\mathcal{G}_{(\X,|||\cdot |||)}$.}
\item[(ii)]{$\mathcal{G}_{(\X,\|\cdot\|)}$ does not admit any non-trivial invariant subspace.}
\end{enumerate}

\subsection{Convex-transitivity of tensor products}

Recall that the projective tensor norm $||\cdot||_{\pi}$ on $\X\bigotimes \Y$ is defined as follows:
\[||u||_{\pi}=\inf \left\{\sum_{i}||x_{i}||_{\X}||y_{i}||_{\Y}:\ u=\sum_{i}x_{i}\otimes y_{i}\right\}.\]
Then $\X\widehat{\bigotimes}_{\pi}\Y$ is the completion of the space $\X\bigotimes \Y$ in the $||\cdot||_{\pi}$ norm.

It is a natural question to ask whether convex-transitivity is preserved in taking tensor products and this problem implicitly appeared in \cite{RT}.

\begin{lemma}\label{lm: tensorbigpoint}
Let $x\in \S_{\X}$ and $y\in \S_{\Y}$ be big points. Then $x\otimes y$ is a big point in $\X\widehat{\bigotimes}_{\pi}\Y$.
\end{lemma}
\begin{proof}
By density considerations it suffices to check that elements of the form $u=\sum_{i}x_{i}\otimes y_{i}$ are contained in 
$\overline{\conv}(\mathcal{G}_{\X\widehat{\bigotimes}_{\pi}\Y}(x\otimes y))$, where $||y_{i}||=1$ for $i$ and 
$\sum_{i}||x_{i}||=1$. Moreover, since the set $\overline{\conv}(\mathcal{G}_{\X\widehat{\bigotimes}_{\pi}\Y}(x\oplus y))$ is invariant under taking convex combinations, it suffices to check that $\{v\otimes w:\ (v,w)\in \S_{\X}\times \S_{\Y}\}$ is contained in $\overline{\conv}(\mathcal{G}_{\X\widehat{\bigotimes}_{\pi}\Y}(x\otimes y))$. By using the fact that $y$ is a big point, we get that $x\otimes w\in \overline{\conv}(\{(\I_{\X}\otimes T)(x\otimes y):\ T\in \mathcal{G}_{\Y}\})$. Similarly, we get that $v\otimes w\in \overline{\conv}(\{(T\otimes \I_{\Y})(x\otimes y):\ T\in \mathcal{G}_{\X}\})$, which gives the claim.
\end{proof}

For convenience we will label a property of normed spaces. We say that $\X$ is $L$-like if the following condition is satisfied: For each $\epsilon>0$, and $x_{1},x_{2},\ldots,x_{n}\in \X$ there are $z_{1},z_{2},\ldots,z_{k}\in\X$ and $L^{1}$-projections 
$P_{1},P_{2},\ldots , P_{k}$ such that 
\[\dist(x_{m}, \span(z_{i}: 1\leq i\leq k))<\epsilon,\]
and $P_{i}z_{i}-z_{i}=P_{i}z_{j}=0$ for $1\leq i\leq k,\ i\neq j$. Recall the well-known fact that $L^{1}$ projections commute.

An example of an $L$-like almost transitive space, apart from $L^{1}$, is its non-complete subspace of functions $f$ 
satisfying $\lim_{t\to 1}f(t)=0$.

\begin{theorem}
Let $\X$ and $\Y$ be convex-transitive normed spaces, and suppose that $\X$ is additionally $L$-like. Then
$\X\widehat{\bigotimes}_{\pi}\Y$ is convex transitive.
\end{theorem}
\begin{proof}
First we will verify that for each $L^{1}$-projection $P$ on $\X$ and rotation $T$ on $\Y$ we have that operators of the form $R=(\I_{\X}-P)\otimes \I_{\Y}+P\otimes T$ are rotations on $\X\widehat{\bigotimes}_{\pi}\Y$. 
Clearly this mapping defines a linear bijection on $\X\bigotimes\Y$, and it is required to check that it is an isometry. In fact, by symmetry it suffices to prove that $R$ is contractive. 
Towards this, given $\epsilon>0$, fix $\sum_{i=1}^{n}x_{i}\otimes y_{i}\in \X\bigotimes\Y$ such that $\sum_{i=1}^{n}\|x_{i}\|\ \|y_{i}\| \leq \|\sum_{i=1}^{n}x_{i}\otimes y_{i}\|+\epsilon$.
Then 
\[\left\|R\left(\sum_{i=1}^{n}x_{i}\otimes y_{i}\right)\right\|\leq \left\|(\I_{\X}-P)\otimes \I_{\Y})\left(\sum_{i=1}^{n}x_{i}\otimes y_{i}\right)\right\|+\left\|(P\otimes T)\left(\sum_{i=1}^{n}x_{i}\otimes y_{i}\right)\right\| \]
\[\leq \sum_{i=1}^{n}\|(\I-P)x_{i}\|\ \|y_{i}\|+\sum_{i=1}^{n}\|Px_{i}\|\ \|y_{i}\|=\sum_{i=1}^{n}\|x_{i}\|\ \|y_{i}\|.\]
Thus we have the claim.

Next, we wish to check that given any norm-$1$ element $v\in \X\widehat{\bigotimes}_{\pi}\Y$ it holds that 
$x\otimes y\in \overline{\conv}(\mathcal{G}_{\X\widehat{\bigotimes}_{\pi}\Y}(v))$. By using multiple times the triangle inequality and the $L$-likeness condition this task reduces to checking the following statement (while retaining the notations of $L$-likeness formulation). Namely, that the closed convex hull of any vector $v_{0}=\sum_{m}\sum_{i} a_{i,m}z_{i}\otimes y_{m}$ of norm-$1$ rotated with rotations type $R$ is the whole unit ball. Here we may assume that 

\[\sum_{i}\|z_{i}\|\  \left\|\sum_{m} a_{i,m}y_{m}\right \| +\epsilon \leq \left\|\sum_{m}\sum_{i} a_{i,m}z_{i}\otimes y_{m}\right\| \quad \mathrm{for}\ i.\] 

By tensor calculus rules we may normalize the previous presentation in such a way that $\|z_{i}\|=1$ for $i$. By the convex transitivity of $\Y$ we get
$\|\sum_{m} a_{i,m}y_{m}\|y\in \overline{\conv}(\mathcal{G}_{\Y}(\sum_{m} a_{i,m}y_{m}))$ for each $i$ and $y\in \S_{\Y}$.
Fix $y\in \S_{\Y}$ and for each $i$ let $T_{1,i},T_{2,i}\ldots,T_{N,i}\in \mathcal{G}_{\Y}$ be rotations such that 
\[\left\|\ \left\|\sum_{m} a_{i,m}y_{m}\right\|y- \frac{1}{N}\sum_{l=1}^{N} T_{l,i}\left(\sum_{m} a_{i,m}y_{m}\right)\right\|\leq \epsilon.\]
Then it can be seen that 
\[S=N^{-1}\sum_{l=1}^{N}((\I_{\X}-\sum_{i}P_{i})\otimes \I_{\Y}+P_{1}\otimes T_{l,1}+P_{2}\otimes T_{l,2}+\ldots+P_{k}\otimes T_{l,k})\]
is an average of $N$ many rotatations, which are obtained by composing $k$ many rotations of type $R$. 
Moreover, 
\[\left\|\sum_{i}z_{i}\otimes \left(\left\|\sum_{m} a_{i,m}y_{m}\right\|y\right) - Sv_{0}\right\|\leq \epsilon.\] 
Write $z=\sum_{i} \|\sum_{m} a_{i,m}y_{m}\| z_{i}$. Note that $z\otimes y\in \overline{\conv}(\mathcal{G}_{\X\widehat{\bigotimes}_{\pi}\Y}(v_{0}))$. 
Since $\|z\|\geq 1-\epsilon$ we obtain by using Lemma \ref{lm: tensorbigpoint} that $\X\widehat{\bigotimes}_{\pi}\Y$ is convex-transitive.
\end{proof}

\subsection{Some 'known' symmetric spaces}

Consider $h^p(B)$ spaces of harmonic functions, where $B$ is the open unit ball of $\R^{n}$. For $1<p<\infty$ there exists by applying the Poisson kernel an isometric isomorphism from $h^{p}(B)$ onto $L^{p}(S)$, where $S$ is the unit sphere, see \cite{SPW}. Similarly, for open half-space $\R^{n}_{+}$ and $1<p<\infty$ the space $h^{p}(\R^{n}_{+})$ is isometric to $L^{p}(\R^{n-1})$. Thus in both cases the harmonic Hardy space is \emph{almost transitive},
being isometric to $L^p$. We do not know if some other common function theory motivated function spaces (ibid. reference) are convex-transitive.
 
Let $\X$ be the space of absolutely continuous functions $f\colon [0,1]\to \R$ with bounded variation and vanishing at the boundary $\{0,1\}$. Then $\X$ endowed with the total variation norm is almost transitive. Indeed, $\X$ is isometric to the $1$-codimensional subspace of $L^{1}$ of all the functions with average $0$. It was pointed out in \cite{Talp} that this subspace is in turn almost transitive, even though it is neither isometric to $L^{1}$, nor contractively complemented in it.


\begin{thebibliography}{10}
\bibitem{Ba}
S. Banach, \emph{Th\'{e}orie des Op\'{e}rations Lin\'{e}aires}, Warsaw (1932).
\bibitem{BR_manuscripta}
J. Becerra Guerrero, A. Rodriguez Palacios, Transitivity of the norm on Banach spaces having a Jordan structure,
Manuscripta Math. 102, (2000), 111-127.
\bibitem{BR2}
J. Becerra Guerrero, A. Rodriguez Palacios, Transitivity of the norm on Banach spaces,
\textit{Extracta Math.} {\bf 17} (2002), 1-58. 
\bibitem{Ca4}
F. Cabello, Transitivity of $M$-spaces and Wood's conjecture
\textit{Math. Proc. Cambridge Phil. Soc. } {\bf 124} (1998), 513-520.  
\bibitem{Ca5}
F. Cabello, Convex transitive norms on spaces of continuous functions, Bull. London Math. Soc. 37 (2005), 107-118.
\bibitem{Ca6}
F. Cabello, Transitivity in spaces of vector-valued functions, Proc. Edinburgh Math. Soc. 53 (2010), 601-608. 
\bibitem{DU}
J. Diestel, J. Uhl Jr, \emph{Vector measures.} Mathematical Surveys, No. 15. American Mathematical Society, Providence, R.I., 1977. 
\bibitem{En} H. B. Enderton, \emph{Elements of set theory}, Academic Press, 1977.  
\bibitem{fhhmz} M. Fabian, P. Habala, P. H\'ajek, V. Montesinos, V. Zizler,
{\it Banach Space Theory, The Basis for Linear and Nonlinear Analysis}, CMS Books in Mathematics, Springer 2011.
\bibitem{FJ1} R.J. Fleming, J.E. Jamison, \emph{Isometries on Banach spaces, function spaces,} Monographs and surveys in pure and applied mathematics 129, Chapman\&Hall, 2003.    
\bibitem{FJ2}R.J. Fleming, J.E. Jamison, \emph{Isometries on Banach spaces, Vector-valued function spaces,} Monographs and surveys in pure and applied mathematics 138, Chapman\&Hall, 2008.  
\bibitem{GKS}
G.Godefroy, N. J.Kalton and P.D. Saphar, \emph{Unconditional ideals in Banach spaces.} 
Studia Math. 104 (1993), 13–59.
\bibitem{GJK}
P. Greim, J.E. Jamison, A. Kaminska, Almost transitivity of some functions spaces,
\textit{Math. Proc. Cambridge Phil. Soc.} {\bf 116} (1994), 475-488.
\bibitem{Heinrich} S. Heinrich, Ultraproducts in Banach space theory, J. Reine Angew. Math. 313 (1980), 72-104.
\bibitem{Kechris} A.S. Kechris, Classical Descriptive Set Theory (Graduate Texts in Mathematics) ({\bf 156}), 
Springer-Verlag, 1995.
\bibitem{LTI}
J. Lindenstrauss, L. Tzafriri, Classical Banach spaces. I. Sequence spaces. Ergebnisse der Mathematik und ihrer Grenzgebiete, Vol. 92. Springer-Verlag, Berlin-New York, (1977).
\bibitem{PR}
A. Pelczynski, S. Rolewicz, Best norms with respect to isometry groups in normed linear spaces, Short Communications
on International Math. Conference in Stockholm (1962), 104.
\bibitem{Rambla}
F. Rambla, A counterexample to Wood's conjecture, \textit{J. Math. Anal. Appl.} {\bf 317} (2006), 659-667.
\bibitem{RT} F. Rambla-Barreno, J. Talponen, Uniformly convex-transitive function spaces,
QJ Math 62 (2011), 189-205. 
\bibitem{Rol}
S. Rolewicz, \emph{Metric Linear Spaces}, Polish Scientific Publishers, Reidel, 1985.
\bibitem{SPW} A. Sheldon, P. Bourdon, R. Wade, \emph{Harmonic Function Theory}, second edition, Graduate texts in mathematics 137. Springer-Verlag, 2001. 
\bibitem{Talp}
J. Talponen, Convex-transitivity in function spaces, J. Math. Anal. 350 (2009), 537-549.
\bibitem{Vlasov}
L.P. Vlasov, Alexandroff's Double Arrow Compact Space and Approximation Theory, Math. Notes 69, (2001), 749-755. 
\bibitem{Wo}
G. Wood, Maximal symmetry in Banach spaces, \textit{Proc. Roy. Irish Acad.} {\bf 82A} (1982), 177-186.
\bibitem{Zizler}
V. Zizler, Non-separable Banach spaces, article in \emph{Handbook of the Geometry of Banach Spaces}, Vol. 2, North-Holland 2003.
\end{thebibliography}
\end{document}